\documentclass{amsart}
\usepackage{amsthm}
\usepackage{graphicx}
\usepackage{enumerate}
\usepackage{amsaddr}
\makeatletter                            
\newtheorem*{rep@theorem}{\rep@title}    
\newcommand{\newreptheorem}[2]{%
\newenvironment{rep#1}[1]{%
 \def\rep@title{#2 \ref{##1}}%
 \begin{rep@theorem}}%
 {\end{rep@theorem}}}
\makeatother

\newtheorem{thm}{Theorem}[section]
\newreptheorem{theorem}{Theorem}

\newtheorem{fact}[thm]{Fact}

\newtheorem{lem}[thm]{Lemma}

\theoremstyle{definition}
\newtheorem{defn}[thm]{Definition}
\theoremstyle{remark}

\newtheorem{claim}{Claim}
\numberwithin{equation}{section}

\newcommand{\set}[1]{\left\{#1\right\}}

\newcommand{\M}{\mathcal{M}}
\newcommand{\N}{\mathcal{N}}

\newcommand{\s}{\mathcal{S}}
\newcommand{\Text}[1]{\text{\textnormal{#1}}}
\newcommand{\spn}{\Text{sp}}

\begin{document}
\setlength{\parskip}{0in}
\title[]{A matroidal generalization of results of Drisko and Chappell}
\author{Daniel Kotlar}
\address{Computer Science Department, Tel-Hai College, Upper Galilee 12210, Israel}
\email{dannykot@telhai.ac.il}
\author{Ran Ziv}
\address{Computer Science Department, Tel-Hai College, Upper Galilee 12210, Israel}
\email{ranziv@telhai.ac.il}
\setlength{\parskip}{0.075in}
\begin{abstract}

Let $\M$ and $\N$ be two matroids on the same ground set. We generalize results of Drisko and Chapell by showing that any $2n-1$ sets of size $n$ in $\M \cap \N$ have a rainbow set of size $n$ in $\M \cap \N$.

\end{abstract}
\maketitle
\section{Introduction}\label{sec1}

An $m\times n$ row-Latin rectangle is an $m\times n$ array in which each of the numbers $1,\ldots,n$ appear exactly once in each row. A \emph{partial transversal of size} $k$ in an $m\times n$ row-Latin rectangle $R$ is a set of $k$ entries of $R$ such that no two of them are in the same row or in the same column. If all the elements in a partial transversal are distinct we call it a \emph{partial rainbow transversal}.

Drisko \cite{Drisko98} proved the following elegant result:

\begin{thm}\label{thm1:1}
Any $(2n-1)\times n$ row-Latin rectangle has a partial rainbow transversal of size $n$.
\end{thm}
In the same paper Drisko also gave an example of a $(2n-2)\times n$ row-Latin rectangle in which there is no partial rainbow transversal of size $n$.

Given sets $F_1, \ldots ,F_m$, a {\em partial rainbow function} is a partial choice function of the sets $F_i$.
A {\em partial rainbow set} is the range of a partial rainbow function. If the sets $F_1, \ldots ,F_m$ are matchings in a given bipartite graph, then a {\em partial rainbow matching} is a partial rainbow set which is also a matching. Drisko's result asserts that any $2n-1$ matchings of size $n$ in a bipartite graph with $2n$ vertices have a partial rainbow matching of size $n$.

In this paper we generalize Drisko's theorem to matroids.

We provide some definitions and notation concerning matroids. For a set $A$ and an element $x$ we use the notation $A+x$ for $A\cup\set{x}$ and $A-x$ for $A\setminus\set{x}$. A collection $\M$ of subsets of a \emph{ground set} $\s$ is a \emph{matroid} if it is hereditary and it satisfies an augmentation property: If $A, B\in\M$ and $|B|>|A|$, then there exists $x\in B\setminus A$ such that $A+x\in \M$. Sets in $\M$ are called {\em independent} and subsets of $\s$ not belonging to $\M$ are called {\em dependent}. 
An element $x\in \s$ is \emph{spanned} by $A$ if either $x\in A$ or $I+x\not\in \M$ for some independent set $I\subseteq A$.
The set of elements that are spanned in $\M$ by $A$ is denoted by $\spn_\M(A)$.
A \emph{circuit} is a minimal dependent set. For more background on matroid theory the reader is referred to in the books of Oxley~\cite{Oxley11} and Welsh~\cite{Welsh76}.

Chappell \cite{Chappell99} proved the following generalization of Drisko's theorem:

\begin{thm}\label{thm1:2}
Any $(2n-1)\times n$ array whose entries are taken from the ground set of a matroid $\M$, and whose rows are independent sets in $\M$, has a partial transversal of size $n$ which is an independent set of $\M$.
\end{thm}

In this paper we prove the following:

\begin{thm}\label{thm1:3}
Let $\M$ and $\N$ be two matroids on the same ground set $\s$. Any $2n-1$ sets of size $n$, each in $\M\cap\N$, have a partial rainbow set of size $n$ in $\M\cap\N$.
\end{thm}

Note that Theorem~\ref{thm1:1} is obtained from Theorem~\ref{thm1:3} by taking both $\M$ and $\N$ to be partition matroids, and Theorem~\ref{thm1:2} follows in the case that one of the matroids is a partition matroid.

\section{Proof of Theorem~\ref{thm1:3}}
We shall use the following basic facts on matroids:

\begin{fact}\label{fact1}
If $I\in\M$ and $I+x\not\in\M$, there exists a unique minimal subset of $I$, which we denote by $C_\M(I,x)$, that spans $x$, and for each $a\in C_\M(I,x)$, we have $I+x-a\in\M$ and $\spn_\M(I+x-a)=\spn_\M(I)$.
\end{fact}

Fact~\ref{fact2} is an immediate consequence of the augmentation property:

\begin{fact}\label{fact2}
Let  $I$ and $J$ be independent sets in $\M$. If $|I|<|J|$, then there exists $J_1\subseteq J\setminus I$ such that $|J_1|\ge |J|-|I|$ and $I\cup J_1\in\M$.
\end{fact}

Fact~\ref{fact3} is also known as the \emph{circuit elimination axiom}:

\begin{fact}\label{fact3}
If $C_1$ and $C_2$ are circuits with $e\in C_1\cap C_2$ and $f\in C_1\setminus C_2$ then there exists a circuit $C_3$ such that $f\in C_3\subseteq (C_1\cup C_2)-e$.
\end{fact}

We shall also need the following lemma from \cite{AKZ14}. The proof is repeated here in order to make the discussion self-contained.

\begin{lem}\label{lem3}
Let $\M$ be a matroid. Let $I\in\M$ and $X=\set{x_1,\ldots,x_k}\subseteq I$ and $Y=\set{y_1,\ldots,y_k}\subseteq \spn_\M(I)\setminus I$ be such that $\spn_\M\left(\left(I\setminus X\right) \cup Y\right)=\spn_\M(I)$.
Suppose $y_{k+1}\in \spn_\M(I)\setminus I$ and $x_{k+1}$ are such that $x_{k+1}\in C_\M(I,y_{k+1})\setminus X$ and $x_{k+1}\not\in C_\M(I,y_i)$ for all $i=1,\ldots,k$. Then $x_{k+1}\in C_\M(\left(I\setminus X\right) \cup Y,y_{k+1})$.
\end{lem}

\begin{proof}
Suppose, for contradiction, that $x_{k+1}\not\in C_\M(\left(I\setminus X\right) \cup Y,y_{k+1})$.
Let $C_1= C_\M(I,y_{k+1})+y_{k+1}$ and $C_2= C_\M(\left(I\setminus X\right) \cup Y,y_{k+1})+y_{k+1}$. Then, by Fact~\ref{fact3}, there exits a circuit $C^1\subseteq C_1\cup C_2$, such that $x_{k+1}\in C^1$ and $y_{k+1}\not\in C^1$.
Since $I$ is independent, $C^1$ must contain at least one element $y_j\in Y$. Let $C_3=C_\M(I,y_j)+y_j$. Since $x_{k+1}\not\in C_\M(I,y_j)$, there exists a circuit $C^2\subseteq C^1\cup C_3$ such that $x_{k+1}\in C^2$ and $y_{j}\not\in C^2$, by Fact\ref{fact3}. Since $C^2\cap Y\subset C^1\cap Y$ we must have $|C^2\cap Y|<|C^1\cap Y|$. We proceed this way until we obtain a circuit whose intersection with $Y$ is empty. This will contradict the independence of $I$.
\end{proof}

\begin{proof}[Proof of Theorem~\ref{thm1:3}]
Let $A_1,\ldots,A_{2n-1}\in \M\cap \N$, each of size $n$. Let $R\in \M\cap\N$ be a partial rainbow set of maximal size. Assume, for contradiction, that $|R|<n$. Without loss of generality we may assume that $R\cap A_i=\emptyset$ for $i=1,\dots,n$. We define,

\begin{defn}
A \emph{colorful alternating trail} (CAT) of length $k$ ($1\le k\le n-1$) relative to $R$ consists of
a set $\{a_1,\ldots, a_k\}$, where distinct $a_i$'s belong to distinct $A_j$'s ($j\in \{1,\ldots,n-1\}$) and a set $\{r_1,\ldots, r_k\}\subset R$, such that
\begin{enumerate}
  \item [(P$_\M$)]
  $R+a_1-r_1+a_2-r_2+\cdots -r_{i-1}+a_i\in \M$ for all $i=1,\ldots,k$.
  \item [(P$_\N$)]
  $R+a_1-r_1+a_2-r_2+\cdots +a_i-r_i\in \N$ and $\spn_\N(R+a_1-r_1+a_2-r_2+\cdots+a_i-r_i)=\spn_\N(R)$ for all $i=1,\ldots,k$.
\end{enumerate}
If, in addition, $R+a_1-r_1+a_2-r_2+\cdots +a_{k-1}-r_{k-1}+a_k\in \N$ then the CAT is called \emph{augmenting} (in this case the condition (P$_\N$) for $i=k$ is redundant).
\end{defn}

Note that since $R$ is of maximal size, no augmenting CAT relative to $R$ exists. The theorem will be proved by showing that the assumption $|R|<n$ yields a partial rainbow set of size $|R|+1$ in $\M\cap\N$.

\begin{claim}\label{claim:1}
For each $k=1,\ldots, n-1$, the CATs involving only elements of $A_1,\ldots,A_k$ contain at least $k$ distinct elements of $R$.
\end{claim}

\begin{proof}[Proof of Claim~\ref{claim:1}]
\renewcommand{\qedsymbol}{}
We prove Claim~\ref{claim:1} by induction on $k$. Since $|R|<n$ and $|A_1|=n$ there exists an element $a_1\in A_1$ such that $R+a_1\in \M$. By the maximality property of $R$ we must have $R+a_1\not\in \N$. By Fact~\ref{fact1}, there exists an element $r_1\in R$ such that $R+a_1-r_1\in \N$ and $\spn_\N(R+a_1-r_1)=\spn_\N(R)$. Thus, $\{a_1\}$ and $\{r_1\}$ form a CAT of length 1 and Claim~\ref{claim:1} holds for $k=1$.

Now suppose Claim~\ref{claim:1} holds for $k-1$ for some $k\ge 2$. Without loss of generality we may assume that the CATs involving only elements of $A_1,\ldots,A_{k-1}$ contain the elements $r_1,\ldots,r_{k-1}\in R$. Let $R_{k-1}=\{r_1,\ldots,r_{k-1}\}$.

Since $|R|<n$, Fact~\ref{fact2} implies that the set $A_k$ contains at least $k$ elements that are not in $\spn_\M(R\setminus R_{k-1})$. Since $|R_{k-1}|=k-1$, at least one of these elements is not in $\spn_\N(R_{k-1})$. Let $a$ be such an element. That is, $a\in A_k$ and
\begin{equation}\label{eq1:1}
    a\not\in\spn_\M(R\setminus R_{k-1})
\end{equation}
and
\begin{equation}\label{eq1:2}
    a\not\in\spn_\N(R_{k-1}).
\end{equation}
First assume $a\not\in\spn_\M(R)$. Then $R+a\in \M$. Since $R$ is of maximal size it follows that $R+a\not\in \N$. By (\ref{eq1:2}), $C_\N(R,a)\not\subseteq R_{k-1}$ and thus, there exists an element $r\in R\setminus R_{k-1}$ such that $R+a-r\in \N$ and $\spn_\N(R+a-r)=\spn_\N(R)$. The CAT consisting of $\{a\}$ and $\{r\}$ contains the extra element $r\not\in R_{k-1}$, and thus Claim~\ref{claim:1} holds for $k$.

Now assume that $a\in\spn_\M(R)$. By (\ref{eq1:1}), there exists $r'\in C_\M(R,a)\cap R_{k-1}$. By the definition of $R_{k-1}$, there exists a CAT containing $r'$, say
\begin{equation}\label{eq1:3}
    \{a_{i_1},\ldots, a_{i_l}\} \Text{ and } \{ r_{i_1},\ldots,r_{i_l}=r'\}.
\end{equation}
We may assume that none of $r_{i_1},\ldots,r_{i_{l-1}}$ is in $C_\M(R,a)$ (otherwise we take $r'$ to be the first of $r_{i_1},\ldots,r_{i_{l-1}}$ that belongs to $C_\M(R,a)$). Let $R'=R+a_{i_1}-r_{i_1}+\cdots-r_{i_{l-1}}+a_{i_l}$. Since $R'\in\M$ and no element of $C_\M(R,a)$ was discarded along the trail from $R$ to $R'$, we have $C_\M(R',a)=C_\M(R,a)$. Hence, $r'\in C_\M(R',a)$. By Fact~\ref{fact1} we have

\begin{equation}\label{eq1:4}
    R'-r'+a\in \M.
\end{equation}

If $a\not\in\spn_\N(R)$ then $R+a\in \N$. By Property (P$_\N$) we have $\spn_\N(R'-r')=\spn_\N(R)$. Thus, $R'-r'+a\in \N$, which, together with (\ref{eq1:4}), yields an augmenting CAT, contrary to the maximality of $|R|$.
Hence, we may assume that $a\in\spn_\N(R)$. By (\ref{eq1:2}), there exists an element $r''\in C_\N(R,a)\setminus R_{k-1}$.

If none of the $a_{i_j}$ ($j=1,\ldots,l$) in (\ref{eq1:3}) satisfies $r''\in C_\N(R,a_{i_j})$, then,  by Lemma~\ref{lem3}, $r''\in C_\N(R'-r',a)$, and thus, the sets $\{a_{i_1}, \ldots, a_{i_l},a\}$ and $\{r_{i_1},\ldots,r',r''\}$ make a CAT which contains the extra element $r''\not\in R_{k-1}$, proving Claim~\ref{claim:1} for $k$.

Otherwise, let $j$ be minimal such that $r''\in C_\N(R,a_{i_j})$ and let $R''=R+a_{i_1}-r_{i_1}+\cdots+a_{i_{j-1}}-r_{i_{j-1}}$. Then, by Lemma~\ref{lem3}, it follows that $r''\in C_\N(R'',a_{i_j})$. Hence, $R''+a_{i_j}-r''\in\N$ and $\spn_\N(R''+a_{i_j}-r'')=\spn_\N(R'')=\spn_\N(R)$. We obtain a CAT consisting of $\{a_{i_1}, \ldots, a_{i_{j-1}}, a_{i_j}\}$ and $\{r_{i_1},\ldots,r_{i_{j-1}},r''\}$. This CAT involves the extra element $r''\not\in R_{k-1}$. Thus, the CATs involving only elements of $A_1,\ldots,A_{k-1},A_k$ contain at least $k$ elements of $R$, proving Claim~\ref{claim:1} for $k$. This completes the proof of Claim~\ref{claim:1}.
\end{proof}

To conclude the proof of Theorem~\ref{thm1:3}, note that by Claim~\ref{claim:1}, the CATs involving only elements of $A_1,\ldots, A_{n-1}$ contain all the elements of $R$. Since $|R|<n$ there exists an element $a_n\in A_n$ such that $R+a_n\in \N$. By the maximality property of $R$, we must have that $a_n\in \spn_\M(R)$. Let $r\in R$ satisfy $R+a_n-r\in \M$.
Applying Claim~\ref{claim:1} to the case $k=n-1$, it follows that $r$ belongs to a CAT consisting of  some sets $\{a_{i_1}, \ldots, a_{i_l}\}$ and $\{r_{i_1},\ldots,r_{i_{l-1}},r_{i_l}=r\}$, where each of $a_{i_1}, \ldots, a_{i_l}$ belongs to a different set among $A_1,\ldots,A_{n-1}$.
Let $R'=R+a_{i_1}-r_{i_1}+\cdots-r_{i_{l-1}}+a_{i_l}$. We may assume that $r$ is the first element in this trail satisfying $r\in C_\M(R,a_n)$ (otherwise we take $r$ to be the first element in the trail belonging to $C_\M(R,a_n)$) and thus, $C_\M(R',a_n)=C_\M(R,a_n)$.
Hence, $R'-r+a_n\in \M$. Since $\spn_\N(R'-r)=\spn_\N(R)$, by Property ($P_\N$), we also have $R'-r+a_n\in \N$. Thus, $R'-r+a_n$ is a rainbow matching of size $|R|+1$.
\end{proof}
\providecommand{\bysame}{\leavevmode\hbox to3em{\hrulefill}\thinspace}
\providecommand{\MR}{\relax\ifhmode\unskip\space\fi MR }
\providecommand{\MRhref}[2]{%
  \href{http://www.ams.org/mathscinet-getitem?mr=#1}{#2}
}
\providecommand{\href}[2]{#2}

\end{document}